\documentclass[12pt]{amsart} \usepackage{amscd}

\newtheorem{theorem}{Theorem}
\newtheorem{lemma}[theorem]{Lemma}
\newtheorem{corollary}{Corollary}
\newtheorem{proposition}[theorem]{Proposition}

\theoremstyle{definition}                 
            \newtheorem{defn}{Definition}
\newtheorem{notation}{Notation}     \newtheorem{conjecture}{Conjecture}
\newtheorem{remark}{Remark}

\usepackage{amsfonts}
\newcommand{\field}[1]{\mathbb{#1}}          \newcommand{\Q}{\field{Q}}
\newcommand{\R}{\field{R}}                   \newcommand{\Z}{\field{Z}}
\newcommand{\C}{\field{C}}

\newcommand{\bs}{\backslash} \newcommand{\ra}{\rightarrow}

\begin{document}

\title[Discrete  Components of  some  Complementary Series]  {Discrete
Components of Some Complementary Series}

\author{B.Speh and T. N. Venkataramana}

\email{venky@math.tifr.res.in}

\subjclass{Primary   11F75;  Secondary   22E40,   22E41\\  B.    Speh,
Department  of  Mathematics,  310,  Malott Hall,  Cornell  University,
Ithaca,  NY  14853-4201, U.S.A.\\  T.   N.   Venkataramana, School  of
Mathematics, Tata Institute of Fundamental Research, Homi Bhabha Road,
Bombay - 400 005, INDIA.  venky@math.tifr.res.in}

\date{}

\begin{center}

\end{center}
\begin{abstract} 

We show  that Complementary Series representations of $SO(n,1)$,   
which   are   sufficiently   close  to   a   cohomological
representation  contain  {\bf  discretely},  Complementary  Series  of
$SO(m,1)$  also sufficiently  close   to   cohomological  
representations,
provided that the degree  of the cohomological representation does not
exceed $m/2$.\\

We prove,  as a consequence, that the  cohomological representation of
degree   $i$  of   the  group   $SO(n,1)$  contains   discretely,  the
cohomological representation  of degree $i$ of  the subgroup $SO(m,1)$
if $i\leq m/2$. \\

As  a global  application,  we show  that  if $G/\Q$  is a  semisimple
algebraic group  such that $G(\R)=SO(n,1)$ up to  compact factors, and
if   we  assume  that   for  all   $n$,  the   tempered  cohomological
representations  are not limits  of complementary  series {\it  in the
automorphic  dual  of  $SO(n,1)$},  then  for  all  $n$,  non-tempered
cohomological representations are isolated  in the automorphic dual of
$G$.   This  reduces conjectures  of  Bergeron  to  the case  of  {\bf
tempered} cohomological representations.
\end{abstract}

\maketitle

\begin{flushright} \end{flushright}

\newpage
\section{Introduction}

A  famous  Theorem  of  Selberg  \cite{Sel}  says,  that  the  non-zero
eigenvalues of the  Laplacian acting on functions on  quotients of the
upper half plane by {\it congruence subgroups} of the integral modular
group are bounded away from  zero, as the congruence subgroup varies.
In fact,  Selberg proved that  every non-zero eigenvalue  $\lambda$ of
the  Laplacian on  functions on  $\Gamma \bs  {\mathfrak  h}$, $\Gamma
\subset SL_2(\Z)$, satisfies the inequality
\[\lambda \geq \frac{3}{16}.\]

A  Theorem  of  Clozel  \cite{Clo}  generalises  this  result  to  any
congreunce quotient of any symmetric space of non-compact type: if $G$
is  a linear  semi-simple  group defined  over  $\Q$, $\Gamma  \subset
G(\Z)$ a  congruence subgroup and  $X=G(\R)/K$ the symmetric  space of
$G$, then nonzero eigenvalues $\lambda  $ of the Laplacian {\it acting
on the space of functions on} $\Gamma \bs X$ satisfy:
\[\lambda  \geq \epsilon,\]  where $\epsilon  >  0$ is  a number  {\bf
independent} of the congruence subgroup $\Gamma$. \\

Analogous  questions on  Laplacians  acting on  differential forms  of
higher degree  (functions may be  thought of as differential  forms of
degree  zero) have  geometric implications  on the  cohomology  of the
locally symmetric space. Concerning  the eigenvalues of the Laplacian,
Bergeron (\cite{Ber}) has conjectured the following:

\begin{conjecture}  (Bergeron) \label{bergeron}  Let $X$  be  the real
hyperbolic  $n$-space  and   $\Gamma  \subset  SO(n,1)$  a  congruence
arithmetic  subgroup.  The  non-zero  eigenvalues $\lambda  $ of  the
Laplacian  acting  on  the  space   $\Omega  ^i  (\Gamma  \bs  X)$  of
differential forms of degree $i$ satisfy:
\[\lambda \geq \epsilon,\] for some $\epsilon >0$ {\bf independent} of
the congruence  subgroup $\Gamma$, provided $i$ is  strictly less than
the ``middle dimension'' (i.e. $i<[n/2]$).\\

If   $n$    is   even,   he conjectures that the   above   conclusion    
holds   even   if $i=[n/2]=n/2$. (When  $n$ is odd,  there is a slightly  
more technical statement which we omit).
\end{conjecture}

In  this paper, we  show, for  example, that  if the  above conjecture
holds true in the {\it middle  degree} for all even integers $n$, then
the conjecture  holds for arbitrary  degrees of the  differential forms
(See Theorem  \ref{laplacian}). For odd  $n$, there is, again,  a more
technical statement (see Theorem \ref{laplacian}, part (2)). \\

The main Theorem  of the present paper is  Theorem \ref{mainth} on the
occurrence of some discrete components of certain Complementary Series
Representations associated  to $SO(n,1)$. The  statement on Laplacians
may  be  deduced  from  the  main  theorem,  using  the  Burger-Sarnak
(\cite{Bu-Sa}) method. \\

We now describe the main  theorem more precisely.  Let $i\leq [n/2]-1$
and $G=SO(n,1)$.   Let $P=MAN$ be the Langlands  decomposition of $P$,
$K\subset G$  a maximal  compact subgroup of  $G$ containing  $M$. Then 
$K$ is isomorphic to the orthogonal group $O(n)$ and $M=O(n-1)$; denote 
by $\mathfrak p$  the standard representation of $O(n)$, tensored by the 
determinant  character;  if ${\mathfrak g}$, $\mathfrak k$ denote the 
complexified Lie algebras of $G$ and $K$, then the Cartan decomposition 
takes the form ${\mathfrak g}={\mathfrak k}\oplus {\mathfrak p}$, as 
$K$-modules.

Let, analogously, ${\mathfrak p}_M$  be the standard  representation 
of $M$ twisted by the  character such that this representation occurs 
in the restriction of ${\mathfrak p}$ to the subgroup $M$ of $K$. 
Let $\wedge ^i$  be its  $i$-th exterior  power representation.   Denote  by $\rho
_P^2$ the character of $P$ acting on the top exterior power of the Lie
algebra of $N$. Consider the induced representation 
(with normalized induction)
\[\widehat {\pi _u(i)}= Ind_P^G (\wedge ^i \otimes \rho _P(a)^u)\] for
$0<u<1-\frac{2i}{n-1}$.   The  representation  $\widehat {\pi  _u(i)}$
denotes  the {\bf  completion} of  the  space $\pi  _u(i)$ of  
$K$-finite functions with respect  to the $G$-invariant metric on  
$\pi _u(i)$, and
is called the  {\bf complementary series representation} corresponding
to the representation $\wedge ^i$ of $M$ and the parameter $u$. \\

Let $H=SO(n-1,1)$ be embedded in $G$  such that $P\cap H$ is a maximal
parabolic  subgroup of  $H$, $A\subset  H$,  and $K\cap  H$ a  maximal
compact  subgroup   of  $H$.   We  now   assume  that  $\frac{1}{n-1}<
u<1-\frac{2i}{n-1}$.       Let     $u'=\frac{(n-1)u-1}{n-2}$;     then
$0<u'<1-\frac{2i}{n-2}$.  Denote by $\wedge ^i _H$ the $i$-th exterior
power of the standard representation of $M\cap H\simeq O(n-2)$.\\

We obtain analogously the complementary series representation
\[\widehat  {\sigma _{u'}(i)}=  Ind_{P\cap H}^H  (\wedge ^i  _H \otimes
\rho _{P\cap H}(a)^{u'}),\]  of $H$. The main theorem  of the paper is
the following.

\begin{theorem}\label{mainth} If
\[\frac{1}{n-1}<u<1-\frac{2i}{n-1},\]  then  the complementary  series
representation  $\widehat{\sigma _{u'}(i)}$  occurs discretely  in the
restriction of  the complementary series  representation $\widehat{\pi
_u(i)}$ of $G=SO(n,1)$, to the subgroup $H=SO(n-1,1)$:
\[\widehat{\sigma _{u'}(i)}\subset \widehat{\pi _u(i)}_{|SO(n-1,1)}.\]
\end{theorem}

\begin{remark} The  corresponding statement is false for  the space of
$K$-finite vectors  in both the  spaces; this inclusion holds  only at
the level of completions of the representations involved.
\end{remark}

As  $u$ tends  to  the limit  $1-\frac{2i}{n-1}$, the  representations
$\widehat{\pi_u(i)}$ tend (in  the Fell topology on the  unitary dual ${\widehat
G}$)  both  to  the   representation  $\widehat{A_i}=\widehat{A_i(n)}$  
and  to  $\widehat{A_{i+1}}=
\widehat{A_{i+1}(n)}$;   we  denote   by  $\widehat{A_j(n)}$   the   
unique  cohomological
representation of $G$ which has cohomology (with trivial coefficients)
in degree $j$ and by $A_j(n)$ the space of $K$-finite vectors in $\widehat{A_j(n)}$.  
Using this,  and the proof of Theorem \ref{mainth}, we
obtain
\begin{theorem}  \label{cohomologicalreps}   The  restriction  of  the
cohomological  representation $\widehat{A_i(n)}$  of $SO(n,1)$  to  the subgroup
$H=SO(n-1,1)$  contains discretely,  the  cohomological representation
$\widehat{A_i(n-1)}$ of $SO(n-1,1)$:
\[\widehat{A_i(n-1)}\subset \widehat{A_i(n)}_{|SO(n-1,1)}.\]
\end{theorem}

Supose now  that $G$ is  a semi-simple linear algebraic  group defined
over $\Q$ and $\Q$-simple, such that
\[G(\R)\simeq  SO(n,1)\] up  to  compact factors  with  $n\geq 4$  (if
$n=7$, we  assume in addition that $G$  is not the inner  form of some
trialitarian  $D_4$). Then  there exists  a  $\Q$-simple $\Q$-subgroup
$H_1\subset G$ such that
\[H_1(\R)\simeq SO(n-2,1),\] up to compact factors.

Denote by $\widehat {G} _{Aut}$ the ``automorphic dual'' of $SO(n,1)$,
in the sense of  Burger-Sarnak (\cite{Bu-Sa}). Suppose $\widehat{A_i}=\widehat{A_i(n)}$ 
is a limit  of representations ${\widehat \rho  _m}$ in ${\widehat G}  _{Aut}$.  The
structure of the unitary dual  of $SO(n,1)$ shows that this means that
$\widehat{\rho _m}=\widehat {\pi _{u_m}(i)}$ for some sequence $\widehat{u_m}$ which tends
from the left, to $1  - \frac{2i}{n-1}$ (or to $1 - \frac{2i+2}{n-1}$;
we   will  concentrate   only  on   the  first   case,  for   ease  of
exposition). \\

Since  $\widehat{\rho  _m}={\widehat  {\pi}}_{u_m}(i)\in  \widehat{G}_{Aut}$,  
a result of Burger-Sarnak (\cite{Bu-Sa}), implies that
\[res(\widehat{\rho _m})\subset \widehat{H}_{Aut}.\]

Applying Theorem \ref{mainth} twice, we get
\[\widehat{\sigma_m}={\widehat    \sigma}_{u_m''}(i)\in    res(\widehat{\rho   _m})\subset
{\widehat H}_{Aut}\]  Taking limits as  $m$ tends to infinity,  we get
$\widehat{A_i(n-2)}$   as    a   limit   of    representations   $\widehat{\sigma_m}$   in
$\widehat{H}_{Aut}$.    Therefore,  the   isolation  of   $\widehat{A_i(n)}$  in
$\widehat{G}_{Aut}$ is reduced to that  for $SO(n-2,1)$, and so on. We
can  finally assume  that  $\widehat{A_i(m)}$ is  a  tempered representation  of
$SO(m,1)$  where  $m=2i$ or  $2i+1$.   \\  This  proves the  following
Theorem.

\begin{theorem}\label{laplacian}  If for all  $m$, the  {\bf tempered}
cohomological representation  $\widehat{A_i(m)}$ (i.e. when $i=[m/2]$)  is not a
limit of  complementary series in  the automorphic dual  of $SO(m,1)$,
then  for  all integers  $n$,  and  for  $i<[n/2]$, the  cohomological
representation  $\widehat{A_i(n)}$  is  isolated  in  the  automorphic  dual  of
$SO(n,1)$.
\end{theorem}

The major  part of the  paper is devoted  to proving the  main theorem
(Theorem \ref{mainth}). The proof  of Theorem \ref{mainth} proceeds as
follows. \\

(1)  We   first  prove  Theorem  \ref{mainth}  when   $i=0$;  that  is
$\widehat{\pi   }_u    =\widehat{\pi   }_u(0)$   is    an   unramified
representation and $\frac{1}{n-1}< u< 1$. In this case, we get a model
of the  representation $\widehat{\pi}_u$ by  restricting the functions
(sections of a line bundle) on  $G/P$ to the big Bruhat cell $Nw\simeq
\R  ^{n-1}$ and  taking their  Fourier transforms.   The $G$-invariant
metric is  particularly easy  to work with  on the  Fourier transforms
(see  Theorem \ref{commutativemodel});  it is  then easy  to  see that
there   is  an   isometric   embedding  $J:   \widehat{\sigma}_{u'}\ra
\widehat{\pi}_u$ (Theorem \ref{unramifiedrestriction}). \\

(2) We now describe the  proof of Theorem \ref{mainth} in the ramified
case. The representation $\widehat{\pi}_u(i)$ is the ``completion'' of
the space  of Schwartz  class functions on  the big Bruhat  cell, with
values in  the representation  $\wedge ^i$ of  the compact  group $M$,
with respect to  a bilinear form. That this bilinear  form is an inner
product  is not  so clear  (in contrast  to the  unramified  case); we
assume this (\cite{Ba-B}; in \cite{Ba-B}, they prove the existence of 
a positive definite inner product on the space of $K$-finite vectors 
in $\widehat{\pi_u(i)}$; since our bilinear form is also $G$-invariant, 
the irreducibility of the module involved shows that these two bilinear 
forms coincide. \\

We again get a model of the representations ${\widehat{\pi }}_u(i)$ on
the space of Fourier transforms of Schwartz class functions on the big
Bruhat    Cell    in   $G/P$.     There    is    an   embedding    $J:
\widehat{\sigma}_{u'}(i)\ra \widehat{\pi}  _u(i)$, which can  be
explicitly described.  We prove  that the
adjoint $J^*$  of $J$  is essentially the  restriction $\pi_{-u}(i)\ra
\sigma_{-u}(i)$. This proves  that the map $J$ is  equivariant for the
action of the  subgroup $H$ and establishes the  main theorem (Theorem
\ref{mainth}) for arbitrary $i$.

\newpage

\section{Preliminary Results and Notation}

\subsection{  The  Group   $G=SO(n,1)$  and  the  representations  $\pi
_u(i)$.}   The  group $G=SO(n,1)$  is  the  subgroup of  $SL_{n+1}(\R)$
preserving the quadratic form
\[x_0x_n+  x_1^2+x_2^2+\cdots+x_{n-1}^2,\]  on  $\R^{n+1}$. Denote  by
$e_0,e_1,\cdots,  e_{n-1},e_n$  the   standard  basis  of  $\R^{n+1}$.
Denote by $w$ the matrix in $SL_{n+1}(\R)$, which takes 
the basis element $e_0$ to $e_n$, $e_n$ to $-e_0$ and takes the other 
basis elements $e_i$ into $e_i$ for $1\leq i\leq n-1$. Then $w$ has the 
matrix form
\[\begin{pmatrix} 0 & 0_{n-1} & 1\cr ^t0 & 1_{n-1} & 0 \cr -1 & 0_{n-1}
& 0 \end{pmatrix}\]  where $0_{n-1}$ is the zero  vector in $\R^{n-1}$
viewed as a row vector of size $n-1$ and $^t0$ refers to its transpose
(a column  vector of  size $n-1$  all of whose  entries are  zero) and
$1_{n-1}$ refers to  the identity matrix of size  $n-1$. The group $G$
is the matrix group
\[\{g\in GL_{n+1}(\R): ^t wg=w\}.\]

Denote by $P$  the subgroup of $G$ which takes  the basis vector $e_0$
into a multiple  of itself. Let $A$ be the  group of diagonal matrices
in $G$, $M$ the maximal compact  subgroup of the centralizer of $A$ in
$G$.  Let $N$ be the unipotent radical of $P$.  Then,
\[A=\{d(a)= \begin{pmatrix}a  &0_{n-1}& 0\cr ^t0 &  1_{n-1}& ^t0\cr 0&
0_{n-1} & a^{-1}
\end{pmatrix} : a\in \R^*\},\]

Similarly $M$ consists of matrices of the form
\[M=\{\begin{pmatrix}  \pm 1 &  0_{n-1} &  0\cr ^t  0 &  m &  ^t0 \cr  0 &
0_{n-1}& \pm 1\end{pmatrix}: m\in SO(n-1)\}.\] The group $N$ is
\[N= \{u(x)= \begin{pmatrix}  1 & x & \frac{\mid  x\mid^2}{2}\cr ^t0 &
1_{n-1}& ^t x\cr 0& 0_{n-1}& 1
\end{pmatrix}: x\in \R^{n-1}\},\] where  $x\in \R^{n-1}$ is thought of
as a row vector.  Then  $P=MAN$ is the Langlands decomposition of $P$.
The map $x\mapsto u(x)$ from  $\R^{n-1}$ into $N$ is an isomorphism of
algebraic groups. \\

Let $K=G\cap O(n+1)$.   Then $K$ is a maximal  compact subgroup of $G$
and  is isomorphic to  $O(n)\times O(1)$.   Let $\rho  _P (a)$  be the
character whose square is the  determinant of $Ad(d(a))$ acting on the
complex  Lie  algebra ${\mathfrak  n}=Lie(N)\otimes  \C$.  Then  $\rho
_P(a)=a^{\frac{n-1}{2}}$.\\

Write ${\mathfrak g}=Lie (G)\otimes \C$, ${\mathfrak k}=Lie (K)\otimes
\C$  and set  ${\mathfrak g}={\mathfrak  k}\oplus {\mathfrak  p}$, the
Cartan Decomposition  which is  complexified.  As a  representation of
$K=O(n)$,   the   space   ${\mathfrak   p}=\C^n$   is   the   standard
representation twisted by a character.   
Denote by  ${\mathfrak  p}_M=\C^{n-1}$ the  standard
representation of $M=O(n-1)$ twisted by that character so that 
${\mathfrak p}_M$ occurs in 
${\mathfrak p}$ restricted to $K$. Let $i$ be an integer with 
$0\leq i\leq [n/2]-1$  where $[~]$  denotes the  integral part.   Denote  
by $\wedge ^i{\mathfrak p}_M$  the representation of  $M$ on the  
$i$-th exterior power of ${\mathfrak p}_M$. \\

If $u\in \C$ and $p=man\in MAN=P$, the map
\[p\mapsto \wedge  ^i {\mathfrak p}_M (m)  \otimes \rho_P(a)^u \otimes
triv  _N(n)=\wedge  ^i  (m)  \rho  _P(a)^u\] is  a  representation  of
$P$. Define
\[\pi  _u(i)=Ind_P^G (\wedge  ^i  {\mathfrak p}_M  \otimes \rho  _P(a)
^u\otimes 1_N)\]  to be the space  of functions on $G$  with values in
the vector space $\wedge ^i {\mathfrak p}_M$ which satisfy
\[f(gman)=\wedge  ^i  (m)  (f(g))\rho  _P(a)^{  -  (1+u)},\]  for  all
$(g,m,a,n)\in G \times  M \times A \times \times  N$. In addition, the
function $f$ is assumed to be  $K$-finite.  The space $\pi _u(i)$ is a
$({\mathfrak g},K)$-module.  The  representation $\pi _u(i)$ is called
the complementary series of $SO(n,1)$ with parameter $u$ and $\wedge ^
i$. \\

When  $i=0$, the representation  $\pi _u(0)$  has a  $K$-fixed vector,
namely the function which  satisfies $f(kan)=\rho _P (a)^{-(1+u)}$ for
all $ kan\in KAN=G$. We denote $\pi _u(0)=\pi _u$; we refer to this 
representation as the  unramified  complementary  series  of  
$SO(n,1)$  with  parameter $u$. \\

\subsection{  Bruhat Decomposition  of  $SO(n,1)$}  We  have the  {\bf
Bruhat Decomposition} of $O(n,1)$:
\begin{equation} \label{bruhatdecomposition} G=NwMAN \cup MAN.
\end{equation} That  is, every matrix  $g\in G=O(n,1)$ which  does not
lie  in  the  parabolic subgroup  $P$,  may  be  written in  the  form
$g=u(x)wmd(a')u(y)$ for some $x,y\in \R^{n-1}$, $a'\in \R^*$ and $m\in
M$  or in  the  form  $g=md(a)u(y)$.  In  these  coordinates, the  set
$NwP=NwMAN$ is  an open subset of  $G$ of full Haar  measure.  The map
$x\mapsto  u(x)$ is  an  isomorphism of  groups  from $\R^{n-1}$  onto
$N$. Let $dn$  denote the image of the  Lebesgue measure on $\R^{n-1}$
under  this  map.  Denote  by  $d^*  a$  the measure  $\frac{da'}{\mid
a'\mid}$ for $a'\in \R^*\simeq A$ via the map $a'\mapsto d(a')$. Under
the Bruhat  decomposition (equation (\ref{bruhatdecomposition}), there
exists a choice of Haar measure on $M$ such that
\begin{equation}\label{haarbruhat}   dg=   dx~d^*   a'~dm~  dy~   \rho
_P(a')^{2}.
\end{equation}

Given $x=(x_1,x_2,\cdots,x_k)\in \R^k$, denote by $\mid x\mid ^2 $ the
sum
\[\mid x\mid ^2 = \sum _{j=1} ^ k x_j ^2.\]

\begin{lemma}  \label{bruhat}  If  $x\in  \R^{n-1}$ and $x\neq 0$, 
then  the  Bruhat decomposition of the element $g= wu(x)w$ is given by
\[g=u(z)w m(x)d(a')u(y)\] with
\[z= \frac{2x}{\mid  x\mid ^2}, ~m= 1_{n-1}  -2 \frac{ ^t  x x}{\mid x
\mid ^2}, ~a'= \frac{\mid x\mid ^2}{2}  ~{\rm and} ~ y= x\] (note that
$^t  x x$  is  a square  matrix  of size  $n-1$ and  that  $m$ is  the
reflection on the orthogonal complement of $x$).
\end{lemma}

The proof  is by  multiplying out the  right hand side  and explicitly
computing the matrix product. \\

We  will   find  it  convenient   to  compute  the  Jacobian   of  the
transformation   $x\mapsto  \frac{2x}{\mid   x\mid   ^2}$  for   $x\in
\R^{n-1}\setminus \{0\}$.
\begin{lemma}\label{reciprocal}    If    $f\in    {\mathcal    C}_c(\R
^{n-1}\setminus \{0\})$ is a continuous function with compact support,
then we have the formula
\[\int _{\R^{n-1}}  dx f(\frac{x}{\mid  x\mid ^2} )=  \int _{\R^{n-1}}
\frac{dx}{(\frac{\mid x \mid ^2}{2})^{n-1}} f(x).\]
\end{lemma}

\begin{proof} In polar  coordinates, $x=rp$ with $r>0$ and  $p$ in the
unit sphere  $S^{n-2}$ of $\R^{n-1}$.  Then the  Lebesgue measure $dx$
has  the form  $dx= r^{n-1}\frac{dr}{r}  d\mu(p)$  where $\mu  $ is  a
rotation invariant  measure on  $S^{n-2}$.  In polar  coordinates, the
transformation  $x\mapsto \frac{2x}{\mid  x\mid ^2  }$ takes  the form
$(r,p)  \mapsto (\frac{2}{r},p)$.   Therefore $d(\frac{x}{\mid  x \mid
^2})=\frac{2^{n-1}}{r^{n-1}}\frac{dr}{r}d\mu(p)=\frac{dx}{r^{2n-2}/2^{n-1}}$.
\end{proof}

\subsection{Iwasawa Decomposition of $O(n,1)$}

We have the {\bf Iwasawa decomposition} of $O(n,1)=G$:
\begin{equation}\label{iwasawadecomposition}   G=KAN.   \end{equation}
That is,  every matrix $g$  in $SO(n,1)$ may  be written as  a product
$g=kd(a)u(y)$ with $k\in K$, $a\in \R ^*$ and $y\in \R^{n-1}$.

\begin{lemma}\label{iwasawa} The  Iwasawa decomposition of  the matrix
$u(x)w$ is of the form
\[u(x)w=k(x)d(a)u(y),\] with $a= 1  + \frac{\mid x\mid ^2}{2}$ and $y=
\frac{x}{a}= \frac{x}{1+\frac{\mid  x\mid ^2}{2}}$. The  matrix $k(x)$
is given by the formula
\[k(x)= u(x)wu(-y)d(a^{-1}).\]
\end{lemma}
\begin{proof} Put $g=u(x)w=kd(a)u(y)$. Compute the product
\[w^tu(x)u(x)w = ^tg g = ^tu(y)d(a)d(a)u(y).\]

By comparing the matrix entries on both sides of this equation, we get
the Lemma.
\end{proof}

We normalize the Haar measures $dg$, $dk$ $d^*a$ and $dy$ on $G$, $K$,
$A$, $N$ respectively so that
\begin{equation} \label{haariwasawa} dg = \rho _P(a)^{2}~dk~d^*a~dy.
\end{equation} ($dy$ is the Lebesgue measure on $N\simeq \R^{n-1}$, on
$A\simeq \R^*$ the measure is $d^*a=\frac{d_1 (a)}{\mid a\mid }$ where
$d_1a$ is the Lebesgue measure).

\subsection{The space $\mathcal {F}_1 $ and a linear form on $\mathcal
{F}_1$}

Define $\mathcal {F}_1$  to be the space of  complex valued continuous
functions $f$  on $G$  such for all  $(g,m,a,n)\in G\times M  \times A
\times N$ we have
\[f(gman)=\rho  _P(a)^{-2}f(g).\]  The  group  $G$ acts  on  $\mathcal
{F}_1$ by  left translations. Denote by ${\mathcal  C}_c(G)$ the space
of  complex valued  continuous  functions with  compact support.   The
following result is well known.
\begin{lemma} \label{surjection} The map
\[f\mapsto  \overline{f}(g)\stackrel{def}{=}\int  _{MAN} f(gman)  \rho
_P(a)^{2}~dm~dn  ~d^*a ,\]  is a  surjection from  ${\mathcal C}_c(G)$
onto the space ${\mathcal F}_1$.
\end{lemma} It  follows from this  and the Iwasawa  decomposition (see
equation (\ref{haariwasawa})), that the map
\begin{equation} \label{invariantlinearform}  L(\overline{f})= \int _K
f(k)dk
\end{equation} is a $G$-invariant  linear form on the space ${\mathcal
F}_1$.

\begin{lemma}  \label{bruhatiwasawa}  For  every  function  $\phi  \in
{\mathcal F}_1$ we have the formula
\[L(\phi)= \int _{\R^ {n-1}} dx \phi (u(x)w)= \int _K dk \phi (k).\]
\end{lemma}
\begin{proof} Given  $f\in {\mathcal  C}_c(G)$, we have  from equation
(\ref{haarbruhat}) that
\[\int   _{\R^{n-1}}   dx  \overline   {f}(u(x)w)=   \int  _G   f(g)dg
=L(\overline{f}).\]    On     the    other    hand,     by    equation
(\ref{haariwasawa}), we have
\[\int_K dk  \overline{f}(k)=\int _G f(g)dg  =L(\overline{f}).\] These
two equations show that the lemma holds if $\phi =\overline{f}$. \\

It follows  from Lemma \ref{surjection} that every  function $\phi \in
{\mathcal  F}_1$ is  of the  form $\overline{f}$  for some  $f$.  This
proves the Lemma.

\end{proof}

\newpage

\section{Unramified Complementary Series for $SO(n,1)$}

\subsection{ A Model for Unramified Complementary Series }

Let  $0<u<1$,  and  $W=\R^{n-1}$.   Denote by  ${\mathcal  S}(W)$  the
Schwartz Space of $W$. \\

Let $K=O(n)$  and $M=O(n-1)$.  Then $K/M$  is the unit  sphere in $W$.
Denote by  $w$ the  non-trivial Weyl group  element of  $SO(n,1)$. The
space $G/P$ is  homogeneous for the action of $K$  and the isotropy at
$P$ is  $M$.  Therefore, $G/P=K/M$.  The  Bruhat decomposition implies
that  we  have an  embedding  $N\subset  K/M$  via the  map  $n\mapsto
nwP\subset  G/P=K/M$.  Denote by  ${\mathcal F}_u$  the space  of {\bf
continuous} functions on $G$ which transform according to the formula
\[f(gman)=f(g)\rho _P(a)^{-(1+u)},\]  for all $g\in G, m\in  M, a\in A
~{\rm  and}~~n\in N$.   An element  of the  Schwartz  space ${\mathcal
S}(W)$ on $W\simeq N$ may be extended to a continuous function on $G$,
which lies in the space ${\mathcal F}_u$.

We will view  ${\mathcal S}(W)$ as a subspace  of ${\mathcal F}_u$ via
this identification.

\begin{lemma}  If   $0<u<1$  and  $\phi   \in  {\mathcal  S}(W)\subset
{\mathcal F}_u$, the integral
\[I_u(\phi)(g)= \int _W dx \phi (gu(x)w)\] converges.

The integral $I(u)(\phi)$ lies in the space ${\mathcal F}_{-u}$.
\end{lemma}

\begin{proof}   The   Iwasawa    decomposition   of   $u(x)w$   (Lemma
\ref{iwasawa}) implies that $u(x)w=kman$ with
\[\rho     _P    (a)^{1+u}=\frac{1}{(1+\frac{\mid     x\mid    ^2}{2})
^{(n-1)(1+u)/2}}\] where $\mid x\mid^2$  is the standard inner product
of $n=u(x)\in N\simeq  W$ in the space $W$ and $k\in  K$, $m\in M$ and
$n\in N$. Therefore, for  large $x\in W$, $\rho _P(a)^{1+u}\simeq \mid
x   \mid  ^{(n-1)(1+u)}$,   which  proves   the  convergence   of  the
integral. The last statement is routine to check.
\end{proof}

There  is  a  $G$-invariant   pairing  between  ${\mathcal  F}_u$  and
${\mathcal F}_{-u}$ defined by
\[<\phi, \psi>=\int  _K dk {\overline  {\phi (k)}}\psi (k),\]  for all
$\psi  \in {\mathcal  F}_u$ and  $\psi \in  {\mathcal  F}_{-u}$.  Here
$\overline {\phi (k)}$ denotes the complex conjugate of $\phi (k)$.\\

Observe that the product $\overline{\phi (g)}\psi (g)$ of the elements
$\phi \in {\mathcal F}_u$ and $\psi \in {\mathcal F}_{-u}$ lies in the
space ${\mathcal F}_1$. Therefore,  $<\phi , \psi>=L(\phi \psi)$ where
$L$ is  the invariant linear form  on ${\mathcal F}_1$  as in equation
(\ref{invariantlinearform}).

From Lemma  \ref{bruhatiwasawa} it  follows, for $\phi  \in {\mathcal
F}_u$ and $\psi \in {\mathcal F}_{-u}$, that
\[<\phi,\psi>  =\int  _{\R^{n-1}}  dx  \overline {\psi  (u(x)w)}  \psi
(u(x)w).\]

Given  $\phi\in  {\mathcal  S}(W)$, denote  by  ${\widehat  \phi}$  its
Fourier transform.

\begin{theorem}  \label{commutativemodel} For  $0<u<1$  and $\phi  \in
{\mathcal S}(W)\subset {\mathcal F}_u$, we have the formula
\[<\phi,  I(u)(\phi)>_{u}= c  \int _W  \mid {\widehat  \phi}(x)\mid ^2
\frac{1}{\mid  x\mid  ^{(n-1)u}} dx,\]  relating  the pairing  between
${\mathcal F}_u$ and  ${\mathcal F}_{-u}$, and a norm  on the Schwartz
space. (here $c$ is a constant depending only on the constant $u$).

In particular, the pairing $<\phi,I(u)(\phi)>$ is positive definite on
the Schwartz space ${\mathcal S}(W)$.
\end{theorem}

The proof of Theorem \ref{commutativemodel} will occupy the rest of the 
section (3.1). Note that Theorem \ref{commutativemodel} is already proved 
in the paper of Grave-Viskik (\cite{GV}). The proof we give below is a little 
different, and moreover, we need the functional equation that is proved below, 
in the course of the proof of Theorem \ref{commutativemodel}. \\

\begin{lemma}  \label{phi}  If  $\phi  ^* \in  {\mathcal  S}(W) \subset
{\mathcal F}_u$, then we have the formula
\[\int _ N  \phi ^* (n'wnw)dn = \int _W \phi  ^*(y+ x) \frac{1}{\mid x
\mid ^{(n-1)(1-u)}}dx,  \] where $n, n'\in  N$ correspond respectively
to $y,x\in W$.

\end{lemma}

\begin{proof} We  view ${\mathcal S}(W)$  as a subspace  of ${\mathcal
F}_u$. Thus, the function $\phi^*(x)\in {\mathcal S}(W)$ is identified
with the section $x\mapsto \phi ^* (u(x)w)$ with $\phi ^*\in {\mathcal
F}_u$. We then get the equality of the integrals:
\[\int _N  \phi^* (nwn'w)dn'=  \int _W \phi  ^* (u(y)wu(x)w)  dx.\] By
Lemma \ref{bruhat},
\[wu(x)w=u(\frac{2x}{\mid    x\mid    ^2})m(x)d(\frac{\mid   x    \mid
^2}{2})u(x).   \] Using  the fact  that $\phi  ^*$ lies  in ${\mathcal
F}_u$, we see that
\[\phi  ^*  (u(y)wu(x)w)=  \phi  ^*  (u(y+\frac{2x}{\mid  x\mid  ^2}))
(\frac{\mid  x\mid ^2}{2})^{-(n-1)(1+u)/2}.\] Therefore,  the integral
of the lemma is
\[\int  _W dx  \phi ^*  (y+\frac{2x}{\mid x\mid  ^2})(\frac{\mid x\mid
^2}{2})^{-(n-1)(1+u)/2)}.\] We now use the computation of the Jacobian
of   the  map   $x\mapsto   \frac{2x}{\mid  x\mid   ^2}$  (see   Lemma
\ref{reciprocal}) to get
\[\int  _N  \phi  ^*  (nwn'w)dn'=  \int _W  \frac{dx}{(\mid  x\mid  ^2
/2)^{n-1}} \phi  ^*(y+x)(\frac{\mid x\mid ^2}{2})^{(n-1)(1+u)/2},\] and
this is easily seen  to be the right hand side of  the equation of the
Lemma.
\end{proof}

Recall that
\[I(u)(\phi  ^*)(nw)=\int  _N  dn'   \phi  ^*  (nwn'w).\]  From  Lemma
\ref{bruhatiwasawa} we have
\[<\phi^*, I(u)(\phi ^*)>= \int _N dn (\phi ^*(nw)I(u)(\phi ^*)(nw)).\]
Therefore, it follows from  the preceding Lemma that the $G$-invariant
pairing $<\phi ^*, I(u)(\phi ^*)>$ is the integral
\[\int  _W dy  \int _W  dx \frac{1}{\mid  x\mid  ^{(n-1)u}} {\overline
{\phi(y)}}  \phi  (y+x)  .\]  The  Fubini Theorem  implies  that  this
integral is
\[\int  _W dx  \frac{1}{\mid x  \mid ^{(n-1)u}}\int  _W  dy {\overline
{\phi  (y)}}\phi (y+x).   \] The  inner  integral is  the $L^2$  inner
product between  $\phi$ and  its translate by  $x$. Since  the Fourier
transform preserves the inner  product and converts translation by $x$
into multiplication by the  character $y\mapsto e^{-2ix.y}$, it follows
that the integral becomes
\[\int  _W  dx \frac{1}{\mid  x\mid  ^{(n-1)(1-u)}}  \int  _W dy  \mid
{\widehat \phi  (y)}\mid ^2  e^{-2i x.y}.\] The  integral over  $y$ is
simply the Fourier transform. Therefore, we have
\begin{equation}        \label{complementaryinnerproduct}       <\phi,
I(u)(\phi)> = \int  _W dx  \frac{1}{\mid  x\mid ^{(n-1)(1-u)}}  {\widehat
{\mid {\widehat \phi (y)} \mid ^2}}(x).
\end{equation}

\begin{lemma}\label{functionalequation}  Let $f \in  {\mathcal S}(W)$,
and  $s$  a complex  number  with real  part  positive  and less  than
$(n-1)/2$. Then we have the functional equation
\[  \Gamma (s)  \int  _ W  dx  \frac{1}{\mid x  \mid ^{2s}}  {\widehat
{f}}(x)  = \Gamma(\frac{n-1}{2} -s)\int  _W dx  \mid x  \mid ^{n-1-2s}
f(x) ,\] where $\Gamma$ is the classical Gamma function.
\end{lemma}

\begin{proof} If $s$ is a complex number with positive real part, then
$\Gamma (s)$ is defined by the integral
\[\int   _0  ^{+\infty}   \frac  {dt}{t}   t^s  e^{-t}.\]   Denote  by
$L(s,\widehat f)$ the integral
\[\int _W dx \frac{1}{\mid  x\mid ^s} {\widehat {f}}(x).\] Multiplying
$L(2s,\widehat f)$ by $\Gamma (s)$, and using Fubini, we obtain
\[\Gamma  (s)L(2s,\widehat  f)= \int  _W  \frac{dx}{\mid x\mid  ^{2s}}
{\widehat {f}}(x)  \int _0 ^\infty e^{-t}  t^s \frac{dt}{t}.\] Changing
the variable $t$ to $\mid x\mid ^ 2t$ and using Fubini again, we get
\[\Gamma (s)L(2s,\widehat f)= \int _0 ^\infty \frac{dt}{t} t^s \int _W
e^{-t\mid x\mid ^2} {\widehat {f}}(x)  dx.\] The inner integral is the
inner product in $L^2(W)$ of the functions $e^ {-t\mid x \mid ^2}$ and
$\widehat  f$.  The  Fourier transform  preserves this  inner product.
Therefore, the  inner product  is the inner  product of  their Fourier
transforms. \\

The    Fourier    transform    of    $e^{-t\mid    x\mid    ^2}$    is
$\frac{1}{t^{(n-1)/2}} e^{- \mid x  \mid ^2/t}$. The Fourier transform
of $\widehat f$ is $f(-x)$. Therefore, we have the equality
\[\Gamma  (s)L(2s, \widehat f)=  \int _0  ^\infty \frac{dt}{t}  t^ {s-
(n-1)/2} \int  _W e^{-\mid x\mid  ^2 / t}  f(-x) dx.\] In  this double
integral, change $t$ to $\mid x\mid ^2 /t$. We then get

\[\Gamma  (s)L(2s,  \widehat  f)   =\int  _0  ^\infty  \frac{dt}{t}  t
^{(n-1)/2 -s} e^{-t} \int _W \mid x\mid ^{2s-(n-1)} f(x) dx .\]

This proves the functional equation.

\end{proof}

Theorem  \ref{commutativemodel}  is now  an  immediate consequence  of
equation  (\ref{complementaryinnerproduct})   and  of  the  functional
equation  in  Lemma  \ref{functionalequation}  (applied to  the  value
$s=\frac{(n-1)(1-u)}{2}$).

\begin{notation}  We  denote  by  ${\widehat  {\pi  _u}  }$  the  {\bf
completion} of the space ${\mathcal S}(W)$ under the metric defined by
Theorem  \ref{commutativemodel}. We  then get  a metric  on ${\widehat
{\pi _u}}$. The $G$-invariance of the pairing between ${\mathcal F}_u$
and ${\mathcal  F}_{-u}$ and  the definition of  the inner  product in
Theorem  \ref{commutativemodel} implies that  the metric  on $\widehat
{\pi  _u}$ is  $G$-invariant.  Note  that elements  of  the completion
${\widehat {\pi _u}}$ may not be measurable functions on $G$, but only
``generalised functions'' or distributions.
\end{notation}

\subsection{Embedding of the unramified Complementary series}
\begin{notation}  Consider  $H=O(n-1,1)\subset  O(n,1)=G$ embedded  by
fixing the  $n-1$-th basis vector $e_{n-1}$. Denote  by $\sigma _{u'}$
the representation  $\pi _{u'}$ constructed  in the last  section, not
for  the  group  $G$,  but,  for  the group  $H$.   We  set  $u'=\frac
{(n-1)u-1}{n-2}$.    We  assume   that  $0<u'<1$   which   means  that
$\frac{1}{n-1}< u< 1$. \\

Denote by $W$ the space  $\R^{n-1}$. Let $W'\subset W$ be the subspace
$\R^{n-2}$ whose last co-ordinate in $\R^{n-1}$ is zero.
\end{notation}

Let  $\widehat{W}_u$ be  the completion  of the  Schwarz space  of $W$
under the metric
\[\mid \phi  \mid _u ^2=  \int _W \frac{dx}{\mid  x\mid^{(n-1)u}} \mid
{\phi   (x)}\mid   ^2.\]  The convergence of this integral near infinity 
is clear since $\phi$ lies in the Schwartz space; the convergence of 
the integral near $0$ follows by converting to polar co-ordinates and 
by using the assumption  that $0<u<1$.\\

Define   the   space   $\widehat{W'}_{u'}$ correspondingly for  $W'$ as 
the  completion of the Schwartz  space of $W'$ under the metric
\[\mid \psi \mid _{u'}^2  =\int _{W'} \frac{dy}{\mid y\mid ^{(n-2)u'}}
\mid \psi (y)\mid ^2.\]

Now, $\widehat{W}_u$  consists of  all measurable functions  $\phi$ on
$W= \R ^{n-1}$ which the $u$-norm, that is, integral
\[\mid \phi \mid ^2_u=  \int _W \frac{dx}{(\mid x\mid )^{(n-1)u}}
\mid  \phi  (x)  \mid ^2  ,\]  is finite.   One may  define
similarly, the $u'$-norm  of a function on $W'$. Given  $\phi $ in the
space $\mathcal  {S}(W)$, we define  the function $J(\phi)$ on  $W$ by
setting.   for  all  $(y,t)\in  W'\times \R=  \R^{n-2}\times  \R\simeq
\R^{n-1} =W$,
\[J(\phi (y,t))= \phi  (y). \] Then $J(\phi )$  is a bounded measurable
function on $\R^{n-1}$.

\begin{proposition}\label{Jisometry}  The  map   $J$  defined  on  the
Schwartz space $\mathcal{S}  (W')$ has its image in  the Hilbert space
$\widehat{W}_u$ and extends  to an isometry from $\widehat{W'}_{u'}\ra
\widehat{W}_u$.
\end{proposition}

\begin{proof}  We  compute  the   $u$-norm  of  the  bounded  function
$J(\phi)$:
\[\mid J(\phi  )\mid ^2_u= \int  _{W'\times \R} \frac{dy~  dt}{(\mid y
\mid ^2 + t^2)^{\frac{n-1)u}{2}}} \mid \phi (y)\mid ^2 .\] By changing
the variable $t$ to $\mid y \mid t$ and using the definition of the map
$J$, we see that this $u$-norm is equal to the integral
\[\int \frac{dt}{(1+t^2)^{\frac{(n-1)u}{2}}} \int _{W'} \frac{dy} {\mid
y \mid  ^{(n-1)u-1}} \mid \phi (y)  \mid ^2. \] The  integral over $t$
converges when the exponent of the denominator term $\frac{(n-1)u}{2}$
is  strictly  larger  than  $1$;  that is,  $u>  \frac{1}{n-1}$.   The
integral  over  $y$  is  simply   the  $u'$  norm  of  $\phi  $  since
$(n-1)u-1=(n-2)u'$.   Therefore,  the  $u$-norm  of $J(\phi)$  is  the
$u'$-norm of $\phi$  up to a constant (the  integral over the variable
$t$ of the  function $\frac{1}{(1+t^2)^{(n-1)u}})$.  We have therefore
proved the proposition.
\end{proof}

The  completion $\widehat{\pi }_u$  (and similarly  
$\widehat{\sigma }_{u'}$)  
was defined with  respect  to  the  metric  $<\phi,I(u)(\phi)>$  (we  
use  Theorem
\ref{commutativemodel} to  say that this  is indeed a metric).   It is
clear from  this definition  that for $\phi  \in \mathcal{S}(W)\subset
\mathcal{F}_u$, the map $\phi  \mapsto \widehat{\phi} $ (the roof over
$\phi$  refers  to  the  Fourier  transform) gives  an  isometry  from
$\widehat{\pi }_u$  onto $\widehat{W}_u$. Similarly  
$\widehat{\sigma }_{u'}$ is
isometric    to     $\widehat{W'}_{u'}$.     Therefore,    Proposition
\ref{Jisometry} says  that the Hilbert  space $\widehat{\sigma }_{u'}$
is isometrically embedded in the Hilbert space $\widehat{\pi }_u$.  We
now show that the map $J$ is equivariant with respect to the action of
the subgroup $H$ on both sides.

\begin{notation}  Let  $\phi  ^* \in  \mathcal{S}(W)\subset  {\mathcal
F}_u$  be a  function whose  Fourier transform  is a  smooth compactly
supported   function   $\widehat{\phi   ^*}$   on   $\R^{n-1}\setminus
\{0\}$. Denote by  $\phi \in \mathcal{F}_{-u}$ the image  of $\phi ^*$
under the map $I(u)$. By the definition of $I(u)$, we have the formula
(see Lemma \ref{phi}),
\[\phi (u(y)w)=\int _{\R^{n-1}} \phi ^* (u(y)wu(x)w) dx.\] Using Lemma
\ref{functionalequation}, we now get
\[\phi       (y)=\int      _{\R^{n-1}}       \frac{dx}{\mid      x\mid
^{(n-1)u}}\widehat{\phi ^*}(x)  e^{-2ix.y}.\] Denote by  ${\mathcal A}_G$ the map
$\widehat{\phi   ^*}   \mapsto   \phi$   from  the   space   $\mathcal
{C}^{\infty}_c(\R^{n-1}\setminus  \{0\})$  into  the  space  $\mathcal
{F}_{-u}$,  by the  preceding formula.   By our  identifications, this
extends  to an isometry  ${\mathcal A}_G$ from  $\widehat{W}_u$ onto  the Hilbert
space $\widehat{\pi }_{-u}\simeq \widehat{\pi  }_u $.  This map exists
for $0< u<1$. The $G$ action on the image $\widehat{\pi }_{-u}$ gives,
via  this  isomorphism ${\mathcal A}_G$,  the  $G$-action  on  the Hilbert  space
$\widehat{W}_u$.  \\

We   similarly  have  a   map  ${\mathcal A}_H$   from  the   space  $\mathcal{C}
^{\infty}_c(\R^{n-2}\setminus \{0\})$ into  $\mathcal {F} ^H _{-u'}$ (of
corresponding functions on the group $H$) which extends to an isometry
from the space $\widehat{W'}_{-u'}$ onto $\widehat{\sigma}_{-u'}\simeq
\widehat{\sigma}_{u'}$.   The $H$  action  on $\widehat{\sigma}_{-u'}$
gives an action of $H$ on the space $\widehat{W'}_{-u'}$. \\

Let $J^*: \widehat{W}_u\ra \widehat{W}_{u'}$ be the adjoint of the map
$J$.
\end{notation}

\begin{proposition} \label{Jres}  We have the formula  for the adjoint
$J^*$ of $J$:
\[J^*= {\mathcal A}_H^{-1}\circ res\circ {\mathcal A}_G. \]
\end{proposition}  (Here   $res  :\mathcal  {F}_{-u}\ra  \mathcal{F}^H
_{-u'}$ is simply restricting the functions on $G$ to the subgroup $H$
(since $(n-1)u-1=(n-2)u'$,  it follows that the  restriction maps the
functions in  $\mathcal{F}_{-u}$ into the  functions in $\mathcal{F}^H
_{-u'}$)).

\begin{proof} The adjoint is defined by the formula
\[<g,  J^*  f>_{W_{u'}}=  <Jg,  f>_{w_u},\] for  all  functions  $f\in
\mathcal{C}^{\infty}_c(\R^{n-1}\setminus     \{0\})$     and     $g\in
\mathcal{C}^{\infty}_c(\R^{n-2}\setminus \{0\})$. We compute the right
hand side  of this formula. By  definition, the inner  product of $Jg$
and $f$ is the integral
\[\int_{\R^{n-1}}  \frac{dx}{\mid  x\mid ^{(n-1)u}}Jg(x)f(x).\]  Write
$x\in  \R^{n-1}=\R^{n-2}\times   \R$  as  $x=(y,s)$.    Then,  by  the
definition of $J$, $Jg(x)= Jg(y,s)=g(y)$ and the preceding integral is
the integral
\[\int _{\R^{n-2}}  dy g(y) \int  ds f(y,s) \frac{1}{(\mid y\mid  ^2 +
s^2)^{(n-1)u}}.\]  By changing  $s$ to  $\mid y\mid  s$,  the integral
becomes
\[\int _{\R^{n-2}} \frac{dy}{\mid y  \mid ^{(n-1)u-1}} g(y) \int _{\R}
\frac{ds}{(1+s^2)^{(n-1)u/2}}f(y,\mid y\mid  s). \] By  the definition
of  the $\widehat{W'}_{u'}$ inner  product, this  is $<g,h>_{W'_{u'}}$
where  $h$ is  the function  in  $y$ defined by 
\[h(y)=\int _{\R} \frac{ds}{(1+s^2)^{(n-1)u/2}} f(y,\mid y\mid s). \]

which  is the  integral over  the
variable  $s$  in  the   preceding  equality.   We  have  proved  that
$<Jg,f>=<g,h>$. \\

By the  definition of the adjoint,  this means that  $h=J^*f$. We have
therefore
\[J^*(f)=\int _{\R}  \frac{ds}{(1+s^2)^{n-1)u/2}} f(y,\mid y\mid s).\]
We now  compute ${\mathcal A}_H  J^*(f)(x)$ at a  point $x\in \R^{n-2}$.   By the
definition of ${\mathcal A}_H$, this is the integral
\[\int_{\R^{n-2}}        \frac{dy}{\mid        y\mid       ^{(n-2)u'}}
J^*(f)(y)e^{-2ix.y}\] The above integral  formula for $J^*$ shows that
${\mathcal A}_HJ^*(f)(x)$ is the integral
\[{\mathcal A}_HJ^*f(x)=\int   _{\R^{n-2}\times   \R}  \frac{dy~ds}{(\mid   y\mid
)^{(n-2)u'}(1+s^2)^{(n-1)u/2}}   f(y,\mid  y\mid  s)   e^{-2i  x.y}.\]
Changing  the   variable  $s$  back  to   $\frac{s}{\mid  y\mid}$  and
substituting  it in   this   integral   (and   using   the   fact   that
$(n-2)u'=(n-1)u-1$), we get
\[{\mathcal A}_HJ^*f(x)= \int _{\R^{n-2}\times \R} \frac{dy~ds} {(\mid y\mid ^2 +
s^2)^{(n-1)u)/2}}f(y,s)  e^{-2i(x,0).(y,s)}.\]  The  latter is  simply
${\mathcal A}_G(f)(x)$ for $x\in \R ^{n-2}$.  By taking $f$ to be ${\mathcal A}_G^{-1}\phi$,
we get,  for all $x\in  \R^{n-2}$, and all  $\phi $ which are  ${\mathcal A}_G$ -
images  of compactly  suported smooth  function  on $\R^{n-1}\setminus
\{0\}$, that
\[{\mathcal A}_HJ^* {\mathcal A}_G^{-1}(\phi)  (x)=\phi (x)=res(\phi)(x).\] This  proves the
Proposition.
\end{proof}

\begin{corollary}\label{Jequivariant}  The  isometric map  $J:\widehat
{W'}_{u'}\ra \widehat{W}_u$  is equivariant for  the action of  $H$ on
both sides.
\end{corollary}
\begin{proof} It is  enough to prove that the adjoint  $J^*$ of $J$ is
equivariant  for $H$  action. It  follows from  the  Proposition, that
$J^*={\mathcal A}_H^{-1}\circ res \circ {\mathcal A}_G$. All the maps 
${\mathcal A}_H$, ${\mathcal A}_G$ and $res$
are $H$-equivariant. The Corollary follows.
\end{proof}

\begin{corollary}\label{rescontinuous}     The     map    $\pi_{-u}\ra
\sigma_{-u'}$ from  the module of $K$-finite  vectors in $\widehat{\pi
}_{-u}$ onto the  corresponding vectors in $\widehat{\sigma}_{-u'}$ is
continuous.
\end{corollary}
\begin{proof} The  map $res$ is,  by the Proposition,  essentially the
adjoint $J^*$ of the map $J$.  But $J$ is an isometry which means that
$J^*$ is a continuous projection. Therefore, $res$ is continuous.
\end{proof}

\begin{theorem}  \label{unramifiedrestriction} Let $\frac{1}{n-1}<u<1$
and $u'=\frac{(n-1)u-1}{n-2}$. The complementary series representation
$\widehat  {\sigma _{u'}}$  of  $SO(n-1,1)$ occurs  discretely in  the
restriction   of   of    the   complementary   series   representation
$\widehat{\pi _u}$ of $SO(n,1)$.
\end{theorem}

\begin{proof}    We   have    the   $H$-equivariant    isometry   from
$\widehat{W}_u'$  into   $\widehat{W}_u$.   Th  former   space  is  by
construction,   isometric  to   $\sigma_{u'}$  and   the   latter,  to
$\widehat{\pi }_u$.

Therefore, the Theorem follows.
\end{proof}

\newpage

\section{Ramified Complementary Series}

In this section, we obtain  an embedding of the ramified complementary
series representation  of $H=SO(n-1,1)$ in  the ramified complementary
series representation  of $G=SO(n-1,1)$.  The method follows  that for
the unramified case,except that in the unramified case, the $G$
-invariant pairing on the  complementary series was directly proved to
be positive  definite, using Theorem  \ref{commutativemodel}. However,
in the ramified case, this seems more difficult to prove.  But, the
positive definitenesss  is actually  known by \cite{Ba-B};  given this
positive  definiteness,  the  proof  proceeds  as  in  the  unramified
case. Nevertheless, we  give complete details, since we will have to deal 
not with scalar  valued operators, but  with matrix valued ones,  and some
complications arise. 

\subsection{  A Model  for Ramified  Complementary Series}  Let $1\leq
i\leq   [n/2]$  be   an   integer.   Set   $W=\R^{n-1}$.   Denote   by
$\mathcal{S}(W,i)$ the  space of functions  on $W$ of  Schwartz class,
with values  in the vector space $\wedge  ^i=\wedge ^i\mathfrak {p}_M$
(recall that $\mathfrak  {p}_M\simeq \C^{n-1}$ is  the standard  
representation of $M=SO(n-1)$ twisted by a character on $O(n-1)$). \\

As we saw before, $G/P=K/M$  and the map $x\mapsto u(x)wP\in G/P$ maps
$W$ injectively onto the ``big  Bruhat Cell'' in $G/P$. Let $u\in \C$.
Denote by $\mathcal {F}_u(i)$  the space of {\bf continuous} functions
on  $G$ with values  in $\wedge  ^i$ which  transform acording  to the
formula
\[f(gman)=    \wedge     ^i(m)(f(g))\rho_P(a)^{-(1+u)},\]    for    all
$(g,m,a,n)\in  G\times  A\times M\times  N$.   \\  An  element in  the
Schwartz space $\mathcal{S}(W,i)$ may be  extended to a section of the
vector  bundle on  $G/P$ associated  to  $\wedge ^i$)  (that is, an 
element of   $\mathcal{F}_u(i)$).    We    denote   this  
embedding    by   $E(u):
\mathcal{S}(W,i)\subset \mathcal{F}_u(i)$.

\begin{lemma}If  $Re(u)>0$  and  $f  \in \mathcal{F}_u(i)$,  then  the
integral
\[I_i(u)(f)(g)= \int _W  dx f(gu(x)w),\] converges and as  a function of
$g\in G$, lies in $\mathcal{F}_{-u}(i)$.
\end{lemma}
\begin{proof}The boundedness of $M$ shows that the convergence amounts
to the convergence (for $Re(u)=C$) of the integral 
\[\int  _W \frac{dx}{(1+\mid  x\mid ^2/2)^{(n-1)(1+C)/2}};\]  
the latter integral converges because of the  assumption $C>0$.

\end{proof}

Given  $f\in \mathcal{S}(W,i)\subset  \mathcal{F}_u(i)$,  we can  then
form  the inner  product  (in the  vector  space $\wedge  ^i$) of  the
vectors $f((u(y)w)$ and $I_i(u)(f)(u(y)w)$.  By an abuse of notation, we
write  $f(y)=f(u(y)w)$.   Then  the  inner  product  is  the  function
$x\mapsto <f(y),  I_i(u)(f)(y)>$. Since $f(y)$ is in  the Schwartz space
and $I_i(u)f(y)$ is  bounded on $W$, it follows  that this inner product
function is integrable over $y$ and thus we get a pairing
\[<f,I_i(u)f>=\int  _W dy <f(y),I_i(u)(f)(y)>.\]  This is  a $G$-invariant
bilinear  form  on  $\mathcal{S}(W,i)$.   We  now  use  a  Theorem  of
\cite{Ba-B}  to  obtain that  if  $0<  u<  \frac{2i}{n-1}$, then  this
pairing is a positive definite inner product on $\mathcal{S}(W,i)$.

Recall  from Lemma  \ref{bruhat}  that if  $x\in  \R^{n-1}$, then  the
Bruhat decomposition of an element of the form
\[wu(x)w=  u(\frac{2x}{\mid  x\mid  ^2})wm(x)d(a)u(x),\]  with  $m(x)=
1-\frac{2^tx x}{\mid  x\mid ^2}\in M$  being the reflection  about the
orthogonal  complement  of  $x$,  and  $a=  \frac{\mid  x\mid  ^2}{2}$
defining the diagonal matrix $d(a)$.

We now prove a Lemma analogous to Lemma \ref{phi}.

\begin{lemma}  \label{ramifiedphi} If $Re(u)>0$  then the  formula for
the  inner  product  is  given,  for $f\in  \mathcal{F}_u$  and  $y\in
\R^{n-1}$, by the formula
\[(I(u)f)(wu(y))= \int _{\R^{n-1}} \frac{dx}{\mid x\mid ^{(n-1)(1-u)}}
m(x)( f(y+x)) .\] In this equation, the function $f$ has values in the
representation space $\wedge  ^i$ of the group $M$  and $m(x)$, 
being an element of $M$,  acts on
the vector  $f(u(y+x)w)$, which, by  an abuse of notation,  is written
$f(y+x)$.
\end{lemma}

\begin{proof} The intertwining operator is
\[I_i(u)f(u(y)w)=\int   _{\R^{n-1}}\frac{dx}{\mid  x\mid  ^{(n-1)(1-u)}}
f(u(y)wu(x)w).\] We now use Lemma \ref{bruhat} to conclude that
\[I_i(u)f(y)= \int_{\R^{n-1}}dx~  \wedge ^i (m(x))(f(y+\frac{2x}{\mid x\mid ^2})).\]
The Lemma now follows from Lemma \ref{reciprocal}, since
\[m(\frac{2x}{\mid x\mid ^2})=m(x):\] the two reflections are the same
since the orthogonal complements are the same.
\end{proof}

In the following we denote by $m_i(x)$ the linear operator $\wedge ^i (m(x))$ 
acting on the representation space $\wedge ^i $ of $M$. \\

We now use Lemma \ref{bruhatiwasawa} to conclude the following.  Given
$f\in   \mathcal{F}_u$,   the   pairing  of   $\mathcal{F}_u(i)$   and
$\mathcal{F}_{-u}(i)$, is given by
\[L(<f(y),I_i(u)f(y)>_{\wedge    ^i})=    \int   _{\R^{n-1}}dy    <f(y),
I_i(u)f(y)>_{\wedge ^i}.\]

Using   the   formula  for   the   intertwining   operator  in   Lemma
\ref{ramifiedphi}, we then have
\begin{lemma}  \label{ramifiedinnerproduct} The  inner product  of $f$
with itself in $\mathcal{F}_u(i)$ is given by
\[<f,f>_{\mathcal{F}_u(i)}=    \int    _{\R^{n-1}}dy~   <f(y),    \int
_{\R^{n-1}}\frac{dx}{\mid x\mid ^{(n-1)(1-u)}} m_i(x)(f(y+x))>.\]
\end{lemma}

If we now assume that $f\in \mathcal{S}(W,i)\subset \mathcal{F}_u(i)$,
then we  see that the  double integral in  the foregoing Lemma  can be
interchanged. The integral of a product of Schwartz class functions is
the integral of  their Fourier transforms. The Fourier  transform of a
translate of a  function by a vector is the  multiple of the transform
of the function by a character.  Using these observations (the Fourier
transform  of  $f(y)$  is  denoted  $\widehat{f}(y)$),  we  have,  for
$\phi\in \mathcal{S}(W,i)$, the formula
\[ <\phi, I_i(u)\phi>_{\mathcal{S}(W,i)}=\]
\[=  \int _{\R^{n-1}}\frac{dx}{\mid  x\mid ^{(n-1)(1-u)}}  \int  dy~ <
\widehat{\phi}(y), m_i(x)(\widehat{\phi}(y)> e^{-2iy.x}.\]

\begin{defn}  Set  $W=\R^{n-1}$ and  $\mathcal{S}(W,i)$  the space  of
$\wedge  ^i$ valued  Schwartz class  functions on  $W$, as  before. 
We denote by $<\phi,  \phi>_{W_u}$, for $\phi \in \mathcal{S}(W,i)$,  the
sesquilinear form
\[<\phi, \phi>_{W_u(i)}=  \int _W \frac{dx}{\mid  x\mid ^{(n-1)(1-u)}}
\int _W dy <\phi(y), m_i(x)(\phi(y)> _{\wedge ^i} e^{-2iy.x}.\]
\end{defn}

Since the map  $\phi \mapsto E(u)(\widehat{\phi})\in \mathcal{F}_u(i)$
defined in section (4.1) is  an injection,  
it follows  from the  result of  \cite{Ba-B} quoted
earlier  that  the  above  sesquilinear  form  is  positive  definite,
provided  $u$  is  real  and  $0<  u<  1-\frac{2i}{n-1}$.   Denote  by
$\widehat{W}_u(i)$  the   {\bf  completion}  of   the  Schwartz  space
$\mathcal{S}(W,i)$ with respect to this inner product. We have then an
isometry $\widehat{\pi}_u(i)\ra  \widehat{W}_u(i)$, which, on Schwartz
class functions  $\phi \simeq E(u)(\phi)$  in $\widehat{\pi}_u(i)$, is
given by the Fourier transform $\phi \mapsto \widehat{\phi}$.

\begin{defn} We define analogously,  the subspace $W'=\R^{n-2}$ and if
$0<u'<1-\frac{2i}{n-2}$,  we get a  metric on  the the  Schwartz space
$\mathcal{S}(W',i)$ given by
\[<\phi,    \phi>_{W'_{u'}(i)}=    \int   _{W'}\frac{dx}{\mid    x\mid
^{(n-2)(1-u')}}\int  _{W'}dy <\phi(y),  m_i(x)(\phi (y)>  e^{-2i x.y}.\]
The completion is denoted $\widehat{W}'_{u'}(i)$. This is isometric to
the representation $\widehat{\sigma} _{u'}(i)$ of $H=SO(n-2,1)$.
\end{defn}

\subsection{Embedding of the Ramified Complementary Series}

\begin{notation} We have the subgroup $H=SO(n-1,1)\subset G=SO(n,1)$; 
the embedding is so that the intersection $M_H=M\cap H$ is the compact 
part of the Levi subgroup of the minimal parabolic subgroup 
$P\cap H$ of $H$; we have the representation 
${\mathfrak p}_{M_H}$ and we have a natural inclusion 
${\mathfrak p}_{M\cap H}\subset {\mathfrak p}_M$ induced by the inclusion 
of the Lie algebra $\mathfrak h$ of $H$ in the Lie algebra 
$\mathfrak g$ of $G$. Consequently, we have the inclusion 
$\wedge ^i {\mathfrak p}_{M_H} \subset \wedge ^i {\mathfrak p}_M$ 
and functions with values in the subspace $\wedge ^i {\mathfrak p}_{M_H}$ 
may be viewed as functions with values in the space $\wedge ^i {\mathfrak p}_M$. \\

We define a map $J: \widehat{W'}_{u'}(i)$ into $\wedge
^i$ valued  continuous functions on $W=\R ^{n-1}=W'\times  \R$, by the
formula
\[J(\phi)(y,t)=\phi(y),\] for all $(y,t)\in \R^{n-2}\times \R=W'\times
\R= W$. Using the convention of the preceding paragraph, we have viewed 
$\phi : W' \ra \wedge ^i {\mathfrak p}_{M\cap H}$ as taking values in the 
larger vector space $\wedge ^i {\mathfrak p}_M$. 
\end{notation}

\begin{lemma}\label{Jramifiedisometry}  Suppose  that  $\frac{1}{n-1}<
u<1-\frac{2i}{n-1}$  and that  $(n-2)(1-u')=(n-1)(1-u)$.  The  map $J$
extends to  an isometry  from $\widehat{W'}_{u'}(i)$ into  the Hilbert
space $\widehat{W}_u(i)$.
\end{lemma}

\begin{proof} It is  enough to check that for $\phi  $ in the Schwartz
space  $\mathcal{W'}_{u'}(i)$,  the   function  $J(\phi)$  has  finite
$\widehat{W}_u(i)$ norm and that its norm  is the norm of $\phi$ as an
element of $\widehat{W'}_{u'}(i)$. \\ We then compute the latter norm.
Set $\lambda = (n-1)(1-u)$.  By assumption, $\lambda =(n-2)(1-u')$. 
Recall that $m_i(x,t)$ is the $i$-th exterior power of the element $m(x,t)\in M$ 
and is viewed as an endomorphism of $\wedge ^i {\mathfrak p}_M$. 
We have
\[      <J\phi,     J\phi>_{W_u(i)}=      \int      _{W'\times     \R}
\frac{dx~dt}{(\sqrt{\mid x\mid ^2+t^2})^\lambda}\]
\[\int  _{W'\times \R}  dy~ ds~  < \phi(y),  m_i(x,t)\phi(y)> e^{-2ix.y}
e^{-2i s.t}. \]

We have  used here the fact that  $J\phi(y,s)=\phi(y)$. The assumption
that $\phi$ is of Schwartz  class implies that all these integrals can
be interchanged, and hence we may write
\[<J\phi,   J\phi>=    \int _{W'\times   W'}   \frac{dx~dy}{\mid   x\mid
^{\lambda}}  <\phi(y), M_i(x)(\phi(y)>_{\wedge  ^i}  e^{-2ix.y},\] where
$M_i(x)$ is the  linear operator on the finite  dimensional vector space
$\wedge ^i$ given by the integral
\[M_i(x)=  \int ds (\mathcal{F}_t  \frac{m_i(x,t)} {(\sqrt{\mid  x\mid ^2+
t^2})^{\lambda}})(s),\]  where  $\mathcal{F}_t$  denotes  the  Fourier
transform with respect to the real variable $t$. \\
Recall that $m_i(x,t)$ is the $i$-th exterior power of the matrix $1-
2\frac{^taa}{\mid a\mid ^2}$ where $a= (x,t)$ is the row matrix of
size $n-1$. Therefore, the product $\Phi_i (x,t)= 
(\mid x\mid ^2+ t^2)^ i  m_i(x,t)$
is a homogeneous polynomial of degree $2i$ 
on the entries of the row matrix $(x,t)$.    

We may write
\begin{equation} \label{M(x)}
M_i(x) \Gamma   (\lambda/2 + i)= \int  _0^{\infty}  \frac{dv}{v}
v^{\lambda /2 + i}     e^{-v} 
\int     _{\R}     ds~     \big{(}\int     _{\R}
\frac{dt}{(\sqrt{\mid         x\mid        ^2+        t^2})^{\lambda}}
\Phi _i (x,t)e^{-2ist}\big{)}.
\end{equation}

Changing $v$ to  $v(\mid x\mid ^2+ t^2)$, we
obtain
\[M_i(x)=\frac{1}{\Gamma (\lambda /2 +i)} 
\int _0^{\infty} \frac{dv}{v}v^{\lambda /2 + i}  e^{-v\mid x\mid ^2} 
\int _{\R}ds~\big{(}\int
_{\R}dt ~ \Phi _i (x,t) e^{-vt^2}e^{-2ist}\big{)}.\]

Set 
\[G_i(s,v)= \int dt~ e^{-vt^2} \Phi _i (x,t) e^{-2\sqrt{-1} st} .\]
Changing $t$ to $\frac{t}{\sqrt{v}}$ i this integral, we obtain that 
\[G_i(s,v)= \frac{1}{\sqrt{v}} \int _{\R} dt~ e^{-t^2}
e^{-2\sqrt{-1}st/\sqrt{v}} \Phi _i(x, t/\sqrt{v}).\]
Thus $G_i(v,s)$ is essentially the Fourier transform of some function
evaluated at $s/\sqrt{v}$. \\

The Fourier transform of $e^{-t^2}$ is itself and the Fourier transform
of the product of $e^{-t^2}$ with a power $t^j$ of $t$  is the $j$-th 
derivative of $e^{-t^2}$. The $j$th derivative is of the form
$e^{-t^2} P_j(t)$ where $P_j$ is a polynomial of degree $j$.  
Since $\Phi _i(x,t)$ is a homogeneous polynomial of degree $2i$, it follows
that for some polynomail $P_i^*$ of degree $i$, we have 
\[G_i(s,v)= \frac{1}{\sqrt{v}} \frac{1}{v^i} P_i^* (s/\sqrt{v}) e^{-s^2/v}.\]
Then we get
\[\int _0 ^{\infty} \frac{dv}{v} v^{\lambda /2+ i} e^{-v\mid x\mid ^2}
  \int ds G_i(v,s)=\]
(after changing $s$ to $s\sqrt{v}$) 
\[= \int _0 ^{\infty} \frac{dv}{v} e^{-\mid x\mid ^2} v^{\lambda /2} 
\int _{\R} ds~ P_i^* (s) e^{-s^2}=\]  
\[= \frac{1}{\mid x\mid ^{\lambda} }\Gamma (\lambda /2)C,\]
where $C$ is some bounded number. 

Therefore, in the triple integral defining $M_i(x)$ in the equation
(\ref{M(x)}), we see that the integrand is an integrable function in
$t,s,v$; therefore its integral over $t$ is an {\bf integrable
function} in $s$ and $v$; the integral over $t$ gives
an integrable function over $s$ whose integral is (by Fourier
inversion), the value of the function at $t= 0$. We therefore get
\[M_i(x)= \frac{1}{\Gamma (\lambda /2 +i)} 
\int _ 0 ^{\infty} \frac{dv}{v} v^{\lambda /2 + i} e^{-v \mid
  x\mid ^2} \mid x\mid ^ {2i} m_i(x,0)= \frac{m_i(x,0)}{\mid x\mid ^{\lambda}},\]    
(the last equality follows by changing $v$ to $v/(\mid x\mid )^2$).

Plugging in this value
for $M_i(x)$ in the expression for $<J\phi, J\phi>$ above, we obtain
\[<J\phi,   J\phi>=  \int   _{W'}   \frac{dx}{\mid  x\mid
^{\lambda}} \int _{W'}dy~ <\phi(y), m_i(x,0)\phi(y)> e^{-2ix.y}.\] 
By the definition
of norm on $\widehat{W'}_{u'}(i)$ this is $<\phi, \phi>_{W'_{u'}(i)}$,
since {\bf by assumption}, $\lambda =(n-1)(1-u)=(n-2)(1-u')$.

Therefore, $J\phi$ has  finite norm and its norm is  equal to the norm
of $\phi$. This proves the Lemma.

\end{proof}

\begin{notation} Denote by ${\mathcal A}_G={\mathcal A}_G(u)$ the 
composite of the two maps:
$\phi    \mapsto   \widehat{\phi}$   from    $\mathcal{S}(W,i)$   into
$E(u)(\mathcal{S}(W,i)\subset    \widehat{\pi}    _u(i)$    and    the
intertwining        map        $I_i(u):        E(u)(\mathcal{S}(W,i))\ra
\mathcal{F}_{-u}(i)$. If  $0< u< 1-\frac{2i}{n-1}$, then  $I_i(u)$ is an
isomorphism,  and  we  get a  metric  on  the  image of  $I_i(u)$  whose
completion is  denoted $\widehat{\pi}_{-u}(i)$; the  map $I_i(u)$ yields
an isometry  $I_i(u): \widehat{\pi }_u(i)\ra  \widehat{\pi}_{-u}(i)$. \\
The composite of the Fourier transform with this isometry $I_i(u)$ gives
an isometry
\[{\mathcal A}_G={\mathcal A}_G(u):  
\widehat{W}_u(i)\ra  \widehat{\pi}_{-u}(i).\] The  
$G$-action on $\widehat{W}_u(i)$ is borrowed from $\widehat{\pi}_u(i)$ via
the  isometry induced by  the Fourier  transform.  Therefore,  the map
${\mathcal A}_G$ is a $G$-equivariant map. \\

We have similarly the isometry
\[{\mathcal A}_H={\mathcal A}_H(u'):   
\widehat{W'}_{u'}(i)\ra  \widehat{\sigma}_{-u'}(i),\]
Similarly, the map ${\mathcal A}_H$ is an $H$ -equivariant isometry.

\end{notation}

\begin{notation}    Given    that   $J$    is    an   isometry    from
$\widehat{W'}_{u'}(i)$  into $\widehat{W}_u(i)$,  denote by  $J^*$ its
adjoint;  $J^*$ is the  orthogonal projection  from $\widehat{W}_u(i)$
onto $\widehat{W'}_{u'}(i)$.

The   functions   in   $\mathcal{F}_{-u}(i)$   have  values   in   the
representation $\wedge ^i {\mathfrak p}_M$ of the group $M=SO(n-1)$ and have
the tranformation property
\[f(gmd(a)n)=  \wedge ^i  (m)(f(g))\rho _P(d(a))^{-(1-u)}.\]  Here, if
$d(a)$  is the  diagonal matrix  in $SO(n,1)$  corresponding  to $a\in
\R^*$ with $a>0$, then $\rho _P(d(a))= a^{n-1}$.  Consequently,
\[\rho _P(d(a))^{-(1-u)}=  a^{(n-1)(1-u)}=a^{(n-2)(1-u')}.\] Denote by
$\wedge ^i {\mathfrak p}_{M_H}$ the representation of $M\cap H= SO(n-2)$ on the
$i$-th exterior  power of  its representation on ${\mathfrak p}_{M\cap H}$,
and  view $\wedge  ^i  {\mathfrak p}_{M_H}$ as  an  $M\cap H$-equivariant  {\it
quotient} of the representation $\wedge ^i{\mathfrak p}_M$ of $M$.  There is 
no conflict with the previous notation where $\wedge ^i {\mathfrak p}_{M\cap H}$ 
was viewed as a {\it subspace} of $\wedge ^i {\mathfrak p}_M$, 
since it ocurs with multiplicity one in the representation 
$\wedge ^i {\mathfrak p}_M$, of $M\cap H$. \\

Consequently,  we  have   a  restriction  map  (actually,  restriction
followed by the above quotient map)
\[res: \mathcal{F}_{-u}(i)\ra \mathcal{F}_{-u'}(i),\]
which is $H$-equivariant. 
\end{notation}
We now prove an analogue of Proposition \ref{Jres}, for the {\it
  ramified} case.

\begin{proposition} \label{Jresramified} Suppose that $\frac{1}{n-1} <
u< 1-\frac{2i}{n-1}$ and  define $u'$ by $(n-1)(1-u)=(n-2)(1-u')$.  We
have the equality
\[{\mathcal A}_H\circ J^*= res\circ {\mathcal A}_G,\] 
with ${\mathcal A}_G={\mathcal A}_G(u)$ and ${\mathcal A}_H={\mathcal A}_H(u')$.
\end{proposition}
\begin{proof} Let $\phi , \psi \in \mathcal{S}(W,i)$.  Since ${\mathcal A}_G$
  is an isometry, we have the equality
\[<\phi, \psi> _{\widehat{W}_u(i)}= <{\mathcal A}_G(\phi), {\mathcal A}_G(\psi)>
_{\mathcal{F}_{-u}(i)}.\]
Now ${\mathcal A}_G(\phi)= I(u)\circ \widehat{\phi}$. If
$<,>_{\mathcal{F}_u\times \mathcal{F}_{-u}}$ denotes the pairing
between $\mathcal{F}_u(i)$ and $\mathcal{F}_{-u}(i)$, then we get 
\[<\phi, \psi>_{\widehat{W}_u(i)}= <\widehat{\phi}, {\mathcal A}_G\psi>
_{\mathcal{F}_u\times \mathcal{F}_{-u}}.\]

By Lemma \ref{bruhatiwasawa}, this is equal to 
\[\int_W dx~ < \phi(x), ({\mathcal A}_G\psi)(x)>_{\wedge ^i},\] which is the $L^2$
inner product of the functions involved. 
We have therefore proved 
\begin{equation} \label{innerG}
<\phi, \psi>_{\widehat{W}_u(i)}= <\widehat{\phi},
({\mathcal A}_G\psi)>_{L^2(W)}. 
\end{equation}
Similarly, we have for the subgroup $H$ and $\phi, \psi\in
\widehat{W'}_{u'}(i)$, the formula

\begin{equation}  \label{innerH}  <\phi, \psi>_{\widehat{W'}_{u'}(i)}=
<\widehat{\phi}_1, ({\mathcal A}_G\psi)>_{L^2(W')},
\end{equation} where  $\widehat{\phi}_1$ is the  Fourier transform for
the subspace  $W'$.  From  equation(\ref{innerG}) we get  for $\phi\in
\mathcal{W'}(i)$,
\[<J\phi,  \psi>_{W_u(i)}= <\widehat{J\phi},  {\mathcal A}_G\psi>_{L^2(W)}=
<\phi, \widehat{{\mathcal A}_G\psi}>_{L^2(W)}= \]
\[= \int _{W'\times  \R} dx~dt <\phi(x),\widehat{({\mathcal A}_G\psi)}(x,t)>.\] The
Fourier transform  on $W'\times \R$  is the Fourier transform  on $W'$
followed   by   the  Fourier   transform   in   the  variable   $t\in
\R$. Moreover,  the integral  of the Fourier  transform in $t$  is the
value of the function at  $0$ (Fourier inversion). This shows that the
last integral is equal to
\[<J\phi , \psi >=  \int _{W'}dx <\phi(x), \widehat{{\mathcal A}_G\psi}_1(x,0)> =
<\widehat{\phi}_1, res\circ  {\mathcal A}_G\psi>_{L^2(W')}.\] On the  other hand,
from the definition of  the adjoint $J^*$ and equation (\ref{innerH}),
we get
\[<J\phi,                    \psi>_{W_u(i)}=                    <\phi,
J^*\psi>_{W'_{u'}(i)}>=<\widehat{\phi}_1,                      {\mathcal A}_H\circ
J^*\psi>_{L^2(W')}.\] The last two equations show that
\[{\mathcal A}_H\circ  J^*\psi=res\circ  {\mathcal A}_G\psi,\]  for all  Schwartz  functions
$\psi $ on $W$. Therefore, the Proposition follows.
\end{proof}

\begin{corollary}    \label{Jequivariantramified}    The    map    $J:
\widehat{W'}_{u'}(i)  \ra  \widehat{W}_u(i)$  is equivariant  for  the
action of $H$ on both sides.
\end{corollary}
\begin{proof}  By the  Proposition \ref{Jresramified},  the  map $J^*$
satisfies ${\mathcal A}_H\circ  J^*=res\circ {\mathcal A}_G$. All the 
maps  ${\mathcal A}_H$, ${\mathcal A}_G$ and
$res$ are equivariant for $H$. The corollary follows.
\end{proof}
\begin{corollary}  \label{rescontinuousramified}  The restriction  map
$res: \mathcal{S}(W_{-u}(i))\ra \mathcal{S}(W'_{-u'}(i))$ is continuous
for  the  $G$  and  $H$  invariant  metrics  on  the  Schwartz  spaces
$\mathcal{S}(W_{-u},i)$ and $\mathcal{S}(W'_{-u'},i)$.
\end{corollary}
\begin{proof}  The map  $res=  {\mathcal A}_H\circ J^*\circ  {\mathcal A}_G^{-1}$; 
the  maps ${\mathcal A}_G,{\mathcal A}_H$  are  isometries  and  $J^*$  
is  an  orthogonal  projection (adjoint  of  an  isometry  $J$).   
Consequently,  $res$  is  also  an
orthogonal projection and is in particular, continuous.
\end{proof}
\begin{theorem}\label{continuityramified}  The  representation on  the
{\bf Hilbert  space} $\widehat{\sigma _{u'}(i)}$  occurs discretely in
the  {\bf  Hilbert  space}  $\widehat{\pi _u(i)}$  if  $\frac{1}{n-1}<
u<1-\frac{2i}{n-1}$.
\end{theorem}

\begin{proof} We have  in fact (Corollary \ref{rescontinuousramified})
that the map
\[res: \pi _{-u}\ra \sigma_{-u'}\] is  continuous for the norms on the
representations involved.

\end{proof}

\newpage
\section{The Main Theorem and Corollaries}

\begin{theorem}\label{maintheorem} Let $n\geq  3$ $H=SO(n-1,1)$ and $G=SO(n,1)$. 
Let $i$ be an integer with $0\leq  i \leq [n/2]-1$ where $[x]$ denotes the
integral part  of $x$.  Let  $\frac{1}{n-1}<u<1-\frac{2i}{n-1}$.  Then
the complementary series representation
\[\widehat{\pi }_u(i)=Ind  _P ^G (\wedge ^i {\mathfrak  p}_M\otimes \rho _P(a)^u
),\] contains discretely, the complementary series representation
\[\widehat{\sigma }_{u'}(i)=Ind_{P\cap  H}^H   (\wedge  ^i  {\mathfrak  p}_{M\cap
H}\otimes  \rho _{P\cap  H}^{u'}),\]  where $u'=\frac{(n-1)u-1}{n-2}$.
Further,   as  $u$   tends  to   $1-\frac{2i}{n-1}$,  $u'$   tends  to
$1-\frac{2i}{n-2}$.
\end{theorem}

\begin{proof} We have already proved this (see Theorem \ref{continuityramified}) 
except for the last statement, which is trivial to prove.   
\end{proof}

\begin{corollary}  Let   $i\leq  [n/2]-1$.   Then   the  cohomological
representation  $A_i(n)$  of $SO(n,1)$  of  degree  $i$ restricted  to
$SO(n-1,1)$  contains  discretely,  the  cohomological  representation
$A_i(n-1)$. \\

The  discrete   series  representation  $A_i(2i)$   of  $SO(2i,1)$  is
contained discretely in the representation $A_i$.
\end{corollary}

\begin{proof}   To   prove   the    first   part,   we   use   Theorem
\ref{maintheorem} and  let $u$  tend to the  limit $1-\frac{2i}{n-1}$.
This  gives: ${\widehat A}_i(n)$ restricted  to $SO(n-1,1)$,  contains 
${\widehat A}_i(n-1)$ discretely. The rest of the corollary follows 
by induction.
\end{proof}

\begin{corollary}  Suppose  that  for  all  $n$,  the  {\bf  tempered}
cohomological representations $A_i(n)$ (i.e. $i=[n/2]$) are not limits
of complementary  series {\bf in  the automorphic dual}  of $SO(n,1)$.
Then, for  any $n$,  all nontempered cohomological  representations of
$SO(n,1)$ are isolated in the automorphic dual.
\end{corollary}

\newpage

\noindent{\bf {Acknowledgments}.} \\
T.N.Venkataramana  thanks  the   Department  of  Mathematics,  Cornell
University and  the Institute for Advanced Study,  Princeton for their
hospitality in 2007- 2008, where  a considerable part of this work was
done.   He  also  thanks   Peter  Sarnak  for  valuable  conversations
concerning the material of  the paper.  He gratefully acknowledges the
support of the J.C Bose fellowship for the period 2008-2013. \\

B. Speh  was partially supported by  NSF grant DMS  0070561. She would
also  like to  thank the  Tata Institute  for its  hospitality  in the
winter of 2009 during a crucial stage of this work.

\end{document}